%% file: conjugate-08-07-2018.tex
\documentclass{svjour3}
\usepackage{amssymb}
\usepackage{multirow}
\usepackage{rotating}
\usepackage{cite}
\usepackage{lscape}
\usepackage{caption}
\usepackage{amsmath}
\usepackage{graphicx}
\usepackage{color}

\font\transp=msbm10 scaled\magstep 1 
scaled\magstep 0
\def\transpR{\transp\char'122}
\def\nameuse#1{\csname#1\endcsname}

\def\Rmv[#1]{\nameuse{@ifnextchar} [{\Rmat{#1}}{\Rvec{#1}}}
\def\Rvec#1{\mbox{\transpR}^{#1}}
\def\Rmat#1[#2]{\mbox{\transpR}^{{#1} \times {#2}}}
\def\@begintheorem#1#2{\par\bgroup{\bf #1\ #2. }\it\ignorespaces}
\def\@opargbegintheorem#1#2#3{\par\bgroup{\bf #1\ #2\ (#3). }\it\ignorespaces}
\def\@endtheorem{\egroup}
\begin{document}
\title{ A conjugate gradient-based algorithm for  large-scale quadratic programming problem with one quadratic constraint}
\titlerunning{}
\author{A. Taati,  M. Salahi }
\institute{A. Taati \and M. Salahi \at
              Faculty of Mathematical Sciences, University of Guilan, Rasht, Iran \\
              Tel.: +98-133333901\\
              Fax: +98-1333333509\\
              \email{ akramtaati@yahoo.com,salahim@guilan.ac.ir}  }         

\maketitle
\hspace{-1cm}\rule{\textwidth}{0.2mm}

\begin{abstract}
In this paper, we  consider the nonconvex quadratically constrained quadratic programming (QCQP) 
with one quadratic constraint. By employing the conjugate gradient method, an  efficient  algorithm  is proposed to solve QCQP that exploits the sparsity of the involved matrices and solves the problem via solving a sequence of positive definite system of linear equations after identifying suitable generalized eigenvalues.  Some numerical
experiments are given to show the effectiveness of the proposed method and to compare it
with some recent algorithms in the literature.
\end{abstract}

\section{Introduction }
We consider the following quadratically constrained quadratic programming (QCQP)
\begin{align}\label{1}
\min \quad q(x):=&x^TAx+2a^Tx \notag\\
g(x):=&x^TBx+2b^Tx+\beta \leq 0, \tag{QCQP}
\end{align}
where $A,B\in \mathbb{R}^{n\times n}$ are symmetric matrices with no definiteness assumed, $a, b\in \mathbb{R}^n$ and $\beta \in \mathbb{R}$.
 When $B=I$, $b=0$ and $\beta <0$,  QCQP reduces to the classical trust-region subproblem (TRS),
  which arises in regularization or trust-region methods for unconstrained  optimization \cite{Conn,8}. Despite being nonconvex,
  numerous  efficient algorithms have been developed to solve   TRS  \cite{lanczos,trs1,trs2,trsei,trs4}. The existing algorithms for TRS can be  classified into two categories;  approximate methods and accurate methods. The Steihaug-Toint algorithm is a well-known  approximate method that
exploits  the preconditioned conjugate gradient iterations for unconstrained minimization of $q(x)$ to obtain an approximate solution to large-scale instances of TRS\cite{Steihaug,toint}.
Precisely, the method follows a piecewise linear path connecting the conjugate gradient iterates for solving the system $Ax=-a$, which either finds an interior solution of TRS  or terminates with  a  point on the boundary which does not allow  the accuracy of the constrained
solution to be specified.   To over come the  lack of accuracy of Steihaug-Toint algorithm, Gould et al. in \cite{lanczos} proposed the generalized Lanczos trust-region  method which minimizes the problem on an expanding sequence of subspaces generated by Lanczos vectors.
 The  classical  algorithm \cite{moretrs} by Mor\'{e} and Sorensen is  an accurate method for TRS, which
 at each iteration makes use of the Cholesky factorization to solve a positive definite linear system and hence is not proper for large-scale instances. A number of algorithms designed for large-scale TRS reformulate the problem as a parameterized eigenvalue problem \cite{trsei,trs2,trsei3}. More recently, it has been shown that TRS can be solved by one generalized eigenvalue problem \cite{trs4}.

 The QCQP is a natural generalization of TRS  and is important in scaled trust-region methods for nonlinear optimization, allowing a possible indefinite scaling matrix. It also
has applications  in double well potential problems \cite{do} and compressed sensing for geological data \cite{do1}.
In recent years,  QCQP has received much attention in the literature and various methods  have been developed to solve it \cite{M16,new,ben,so,wol,salahi}.
In  \cite{M16},  Mor\'{e} gives a
  characterization of the global minimizer of  QCQP and describes
 an  algorithm for the solution of QCQP which extends the one for TRS \cite{moretrs}.
A connection between the solution
methods for QCQP and semidefinite programming (SDP) is
established in \cite{new}. However, the SDP approach is not practical for large-scale problems.
Recently, it has been  shown that when the quadratic forms are simultaneously diagonalizable, QCQP  has a  second order cone programming  (SOCP) reformulation which is significantly more tractable than a semidefinite problem \cite{ben}. In \cite{so}, the authors  further showed that QCQP with
an optimal value bounded from below is  SOCP representable, which extends the results in \cite{ben}.  The QCQP with the two-sided constraint has been studied in \cite{wol}. In particular, it has been shown that QCQP can be transformed into a
  parameterized generalized
eigenvalue problem \cite{wol}.  Salahi and Taati \cite{salahi} also
derived a diagonalization-based algorithm under the simultaneous diagonalizable condition of the
quadratic forms. The method  is  proper for small and medium-scale instances since it  needs matrix factorization.
  Recently,  Adachi and Nakatsukasa
developed an algorithm for
QCQP that requires finding just one eigenpair of a generalized eigenvalue problem \cite{qeig}, extending the one for TRS \cite{trs4}.
 However, due to the requirement of constructing explicit form of the Hessian matrices, when the involved matrices are not highly sparse, the method is not efficient for large-scale problems.
 Most recently, in \cite{novel}, the authors have derived a novel convex quadratic reformulation for QCQP and established that the optimal solution of QCQP can be recovered from the optimal solution of the new reformulation. They also developed two efficient algorithms to solve the new reformulation.

In the present paper, we propose a conjugate gradient-based algorithm for QCQP
which applies the conjugate gradient method to  solve  a sequence of positive definite system of linear equations, and hence  is efficient for large sparse instances of QCQP.
 It is worth noting that  although  the method of \cite{M16} and  ours  follow the same framework to solve the problem
 in the case where the optimal
Lagrangian Hessian matrix is positive definite,
  there are some differences in our approach. Firstly, in \cite{M16}, no specific algorithm is proposed to solve the so-called  secular equation while here we propose an efficient algorithm to solve it. Secondly,  our approach first verifies hard case (case 2) and if it is detected, the optimal solution of QCQP is computed via solving a positive definite system. We also emphasize that our approach for verifying hard case 2 is different from the one presented in \cite{wol}.

The rest of the paper is organized as follows. In Section 2, we review some known facts
 related to QCQP.  A conjugate-gradient based  algorithm is introduced in Section 3. Finally, in Section 4, we give some numerical results to show the effectiveness of the proposed method comparing with the algorithms from \cite{qeig} and \cite{novel}.\\ \\
 {\bf Notations:}  Throughout this paper, for two symmetric matrices  $A$ and $B$, $\lambda_{\text{min}}(A,B)$ denotes the smallest generalized eigenvalue
of  pencil $(A,B)$, $A \succ 0(A\succeq 0) $ denotes $A$ is
positive definite (positive semidefinite),
 $A^{\dag}$ denotes the Moore-Penrose generalized inverse of $A$, $\text{Range}(A)$  and $\text{Null}(A)$ denote
its Range and Null spaces, respectively. Finally, for an interval $I$,  $\text{int}(I )$ stands for
 its interior.

\section{Preliminaries: optimality, easy and hard cases}
In this section, we state some results related to the QCQP that are useful in the next section. We start by considering the following assumptions.\\
{\bf Assumption 1.}
 QCQP satisfies the Slater's condition, i.e., there exists $\hat{x}$ with $g(\hat{x})<0$.\\
 {\bf Assumption 2.} There exists $\hat{\lambda}\geq 0$ such that $A+\hat{\lambda} B\succ 0$.\\
When  Assumption 1 is violated,   QCQP reduces to an unconstrained minimization  problem and Assumption 2 ensures that  QCQP has an optimal solution \cite{qeig}.
The following theorem gives a set of necessary and sufficient conditions for the global optimal solution  of QCQP under Assumption 1.\\
\begin{theorem}[\cite{M16}]\label{ttt1}
Suppose that QCQP satisfies Assumption 1. Then $x^{*}$ is a global optimal solution if and only if there exists $\lambda^{*}\geq 0$ such that
\begin{align}
&(A+\lambda^{*} B)x^{*}=-(a+\lambda^{*}b),\label{op1}\\
&g(x^{*})\leq 0,\label{op2}\\
& \lambda^{*} g(x^{*})=0,\label{op3}\\
&(A+\lambda^{*}B)\succeq 0.\label{op4}
\end{align}
\hspace{11.5cm} $\Box$
\end{theorem}
An optimal  solution of  \ref{1} belongs to one of the following two types:  an interior solution with $g(x^{*})<0$ or a boundary solution with $g(x^{*})=0$.
The following lemma states the case where  QCQP has no boundary solution.
\begin{lemma}\label{ll1}
The QCQP has no  optimal solution on the boundary of the feasible region if and only if $A\succ 0$ and $g(-A^{-1}a)<0$.
\end{lemma}
\begin{proof}
Let $x^{*}$ be the optimal solution of  QCQP. If $A\not \succeq 0$, then from Theorem \ref{ttt1}, $g(x^{*})=0$ and hence the optimal solution is on the boundary of the feasible region. Now suppose that $A$ is positive semidefinite (singular) and $g(x^{*})<0$. Let $v\in \text{Null}(A)$ and consider the following quadratic equation of variable $\alpha$:
$$g(x^{*}+\alpha v)=(x^{*}+\alpha v)^TB(x^{*}+\alpha v) +2b^T(x^{*}+\alpha v)+\beta=0.$$
This equation has a root since  $v^TBv>0$ and $g(x^{*})<0$. To show $v^TBv>0$, recall that, by Assumption 2, there exists $\hat{\lambda}\geq 0$ with $A+\hat{\lambda}B\succ 0$. This implies that $v^TAv+\hat{\lambda}v^TBv>0$. Since $v^TAv=0$, then $v^TBv>0$. The discussion above proves that in the case where $A$ is positive semidefinite (singular), QCQP has a solution on the boundary. The only case that must be considered is the case where $A$  is positive definite. In this case, $x^{*}=-A^{-1}a$ is the unique unconstrained minimizer of $q(x)$. Hence, QCQP has no solution with $g(x^{*})=0$ if and only if $g (-A^{-1}a)<0$. \hspace{9.5 cm} $\Box$
\end{proof}
In view of Lemma \ref{ll1}, unless $A\succ 0$ and $g(-A^{-1}a)<0$,  QCQP has a boundary solution. Hence, from now on, we focus on the boundary solutions since an interior solution that can be obtained by solving the linear system $Ax=-a$.

By Assumption 2, $A$ and $B$ are simultaneously diagonalizable \cite{co}, i. e., there exists an invertible matrix $Q$  and diagonal
matrices $ D=\text{diag}(d_1,\cdots,d_n)$ and $E=\text{diag}(e_1,\cdots,e_n)$ such that
$Q^T AQ = D$ and $ Q^T  BQ = E$. It is easy to see that $A+\lambda B\succeq 0$ if and only if $\underline{\lambda}\leq \lambda \leq \bar{\lambda}$ where
$$\underline{\lambda}=\max\{-\frac{d_i}{e_i}|e_i>0\}, \quad \bar{\lambda}=\min\{-\frac{d_i}{e_i}|e_i<0\}.$$
If $A\succ 0$, $\underline{\lambda}<0$ else $\underline{\lambda}\geq 0$.
Let $I$ be the smallest interval containing all nonnegative $\lambda$ satisfying   the optimality conditions (\ref{op1}) and (\ref{op4}). It follows from Theorem \ref{ttt1} that the interior of $I$,   $\text{int}(I)=(\max\{0,\underline{\lambda}\},\bar{\lambda}) $. The interval $I$  contains $\bar{\lambda}$ as the right  endpoint if $(a+\bar{\lambda}b)\in \text{Range}(A+\bar{\lambda} B)$ and it contains $\underline{\lambda}$ as the left endpoint if $\underline{\lambda}\geq 0$
 and $(a+\underline{\lambda}b)\in \text{Range}(A+\underline{\lambda} B)$.
For any $\lambda\in (\underline{\lambda},\bar{\lambda})$, define
\begin{align*}
\phi(\lambda):=x(\lambda)^T B x(\lambda)+2b^Tx(\lambda)+\beta,
\end{align*}
where $x(\lambda)=-(A+\lambda B)^{-1}(a+\lambda b)$.
The function $\phi(\lambda)$ has the following properties.
\begin{lemma}[\cite{M16}]\label{llll1}
$\phi(\lambda)$ is either constant or strictly decreasing on  $(\underline{\lambda},\bar{\lambda})$. $\phi(\lambda)$ is constant if and only if
\begin{align}\label{12}
 \begin{bmatrix}a\\b
\end{bmatrix}\in\text{Range}
\begin{bmatrix}
A\\B
\end{bmatrix}.
\end{align}
\end{lemma}
 \begin{proposition}[\cite{M16,qeig}]\label{ppr1}
    Suppose that $\bar{\lambda}=+\infty$. Then $\lim_{\lambda\rightarrow +\infty}\phi(\lambda)<0$.
    \end{proposition}
The QCQP similar to TRS is classified into  easy case and hard case (case 1 and 2) instances.  The characterization of the easy and hard
cases of the QCQP with the two-sided constraint is given in \cite{wol}.  Here we adapt the characterization  for the QCQP based on  the requirement $\lambda^{*}\geq 0$ where $\lambda^{*}$ denotes the optimal Lagrange multiplier. We have to consider separately
the cases where $A$ is positive definite and $A$ is not positive definite as follows:\\
\begin{itemize}
\item  {\bf $A$ is positive definite.}
\begin{enumerate}
\item The easy case occurs if one of the following holds:
\begin{itemize}
\item[(i)]
$\overline{\lambda}$ is infinite.
\item[(ii)] $\overline{\lambda}$ is finite and $a+\overline{\lambda}b \notin  \text{Range}(A+\overline{\lambda}B)$. This implies that $\lambda^{*}\in [0 \, \overline{\lambda})$.
\end{itemize}
\item The hard case 1 occurs if the following  holds:
\begin{itemize}
    \item[(i)]
      $\overline{\lambda}$ is finite and $a+\overline{\lambda}b \in  \text{Range}(A+\overline{\lambda}B)$ and   $\lambda^{*}\in[0 \, \overline{\lambda})$.
     \end{itemize}
\item
The hard case 2 occurs  if the following holds:
\begin{itemize}
    \item[(iii)]$\overline{\lambda}$ is finite and $a+\overline{\lambda}b \in  \text{Range}(A+\overline{\lambda}B)$ and $\lambda^{*}=\overline{\lambda}$.
    \end{itemize}
    \end{enumerate}
    \item {\bf $A$ is not positive definite.}
    \begin{enumerate}
\item The easy case occurs if one of the following three cases holds:
\begin{itemize}
\item[(i)]
Both $\underline{\lambda}$ and $\overline{\lambda}$ are finite,  $a+\underline{\lambda}b \notin \text{Range}(A+\underline{\lambda} B)$ and  $a+\overline{\lambda}b \notin \text{Range}(A+\overline{\lambda} B)$. This implies that $\lambda^{*}\in(\underline{\lambda}, \, \overline{\lambda})$.
\item[(ii)] Only $\underline{\lambda} $ is finite and $a+\underline{\lambda}  b \notin  \text{Range}(A+\underline{\lambda}B)$. This implies that $\lambda^{*}\in(\underline{\lambda}, \, \infty)$.
\item[(iii)] Only $\overline{\lambda}$ is finite and $a+\overline{\lambda}b \notin  \text{Range}(A+\overline{\lambda}B)$. This implies that $\lambda^{*}\in(-\infty, \, \overline{\lambda})$.
\end{itemize}
\item The hard case 1 occurs if one of the following three cases holds:
\begin{itemize}
\item[(i)]
Both $\underline{\lambda}$ and $\overline{\lambda}$ are finite and $a+\underline{\lambda}b \in \text{Range}(A+\underline{\lambda} B)$ or $a+\overline{\lambda}b \in \text{Range}(A+\overline{\lambda} B)$ and $\lambda^{*}\in(\underline{\lambda}, \, \overline{\lambda})$.
\item[(ii)]Only $\underline{\lambda} $ is finite and $a+\underline{\lambda}  b \in \text{Range}(A+\underline{\lambda}B)$ and $\lambda^{*}\in(\underline{\lambda}, \, \infty)$.
    \item[(iii)]
     Only $\overline{\lambda}$ is finite and $a+\overline{\lambda}b \in  \text{Range}(A+\overline{\lambda}B)$ and   $\lambda^{*}\in(-\infty, \, \overline{\lambda})$.
     \end{itemize}
\item
The hard case 2 occurs if one of the following holds:
\begin{itemize}
\item[(i)]Both $\underline{\lambda}$ and $\overline{\lambda}$ are finite and $a+\underline{\lambda}b \in \text{Range}(A+\underline{\lambda} B)$ and $\lambda^{*}=\underline{\lambda}$ or $a+\overline{\lambda}b \in \text{Range}(A+\overline{\lambda}B)$ and $\lambda^{*}=\overline{\lambda}$.
    \item[(ii)]
    Only $\underline{\lambda} $ is finite and $a+\underline{\lambda}  b \in  \text{Range}(A+\underline{\lambda}B)$ and $\lambda^{*}=\underline{\lambda}$.
    \item[(iii)]Only $\overline{\lambda}$ is finite and $a+\overline{\lambda}b \in  \text{Range}(A+\overline{\lambda}B)$ and $\lambda^{*}=\overline{\lambda}$.
    \end{itemize}
    \end{enumerate}
    \end{itemize}
   From Lemma \ref{llll1} and Proposition \ref{ppr1}, we see that in the easy case and hard case 1, $\lambda^{*}$
is the unique solution  of equation $\phi(\lambda)=0$ in  the interval  $(\max\{0,\underline{\lambda}\},\bar{\lambda})$.  Hard case 2 corresponds to the case where  equation $\phi(\lambda)=0$ has no solution in $(\max\{0,\underline{\lambda}\},\bar{\lambda})$.
\section{A conjugate gradient-based algorithm}
In this section,
we assume that a value of $\hat{\lambda}$  such that $A+\hat{\lambda}B\succ 0$ is known. By
Lemma \ref{ll1}, excluding the case where $A$ is positive definite and $g(-A^{-1}a)<0$, solving  QCQP is equivalent to find a nonnegative $\lambda^{*}\in [\max\{0,\underline{\lambda}\},\bar{\lambda}]$ such that $\phi(\lambda^{*})=0$.
   Note that except the case where QCQP is a hard case 2 instance, the nonlinear equation $\phi(\lambda)=0$ has a unique root in interval $(\max\{0,\underline{\lambda}\},\bar{\lambda})$. According to this fact, we propose a conjugate gradient-based algorithm for  QCQP which first checks for hard case 2 and if this is the case, a global optimal solution of QCQP is obtained via solving a positive definite system  of linear equations, otherwise, the global solution of QCQP is computed by finding the root of   equation $\phi(\lambda)=0$ in $(\max\{0,\underline{\lambda}\},\bar{\lambda})$.
We use the value of $\phi(\hat{\lambda})$ to form the algorithm. We consider the following three cases:
\begin{enumerate}
\item {\bf Case 1. }$ \phi(\hat{\lambda})=0$.\\
In this case,  obviously  $(\hat{\lambda},x(\hat{\lambda}))$ satisfies the optimality conditions (\ref{op1}) to (\ref{op4}) and hence $x(\hat{\lambda})$ is the unique optimal solution of QCQP.

\item {\bf Case 2. } $\phi(\hat{\lambda})>0.$\\
In this case, by Lemma \ref{llll1}, $\lambda^{*}\in (\hat{\lambda},\bar{\lambda}]$. Now we consider the following two possible subcases:
\begin{enumerate}
\item {\bf $\bar{\lambda}$ is infinite.} In this case, from Proposition \ref{ppr1},  equation $\phi(\lambda)=0$   has a unique root in $(\hat{\lambda},\bar{\lambda})$.
\item {\bf $\bar{\lambda}$ is finite.} In this case, either QCQP is a hard case 2 instance ($\lambda^{*}=\bar{\lambda}$) or equation $\phi(\lambda)=0$   has a unique root in $(\hat{\lambda},\bar{\lambda})$.
\end{enumerate}
\item {\bf Case 3. } $\phi(\hat{\lambda})<0.$\\
In this case,  by Lemma \ref{llll1}, $\lambda^{*}\in [\max\{0,\underline{\lambda}\},\hat{\lambda})$. Now we consider the following two possible subcases:
\begin{enumerate}
 \item If $\underline{\lambda}<0$, then $A\succ 0$. Set $x^{*}=-A^{-1}a$. If $g(x^{*})\leq 0$, then $x^{*}$ with $\lambda^{*}=0$ is the optimal solution of QCQP. Otherwise,  equation $\phi(\lambda)=0$   has a unique root in $(0,\hat{\lambda})$.
\item If $\underline{\lambda}\geq 0$,  either QCQP is a hard case 2 instance ($\lambda^{*}=\underline{\lambda}$) or  equation $\phi(\lambda)=0$   has a unique root in $(\underline{\lambda},\hat{\lambda})$.
\end{enumerate}
\end{enumerate}
To the best of our knowledge,  there is no  algorithm  in the literature to compute $\underline{\lambda}$ and $\bar{\lambda}$  for general $A$ and $B$ when the value of $\hat{\lambda}$ is not known. However, when the value of $\hat{\lambda}$ is available, $\underline{\lambda}$ and $\bar{\lambda}$
can be efficiently computed via finding some generalized eigenvalues of a matrix pencil \cite{wol,novel}. Precisely,
$A+\lambda B\succeq 0$ if and only if $ \underline{\lambda}\leq \lambda \leq \bar{\lambda}$ where $\underline{\lambda}=\lambda_1+\hat{\lambda}$, $\bar{\lambda}=\lambda_2+\hat{\lambda}$,
$$\lambda_1=\begin{cases}
\frac{1}{\lambda_{\text{min}}(-B,A+\hat{\lambda}B)}& if\,\, \lambda_{\text{min}}(-B,A+\hat{\lambda}B)<0,\\
-\infty & otherwise,
\end{cases}$$
and
$$\lambda_2=\begin{cases}
\frac{1}{\lambda_{\text{min}}(B,A+\hat{\lambda}B)}& if \,\,\lambda_{\text{min}}(B,A+\hat{\lambda}B)<0,\\
\infty & otherwise.
\end{cases}$$
It is worth noting that, when $\hat{\lambda}$ is given, we only need to compute one extreme eigenvalue to determine the initial interval containing the optimal Lagrange multiplier, i.e., only one of $\underline{\lambda}$ and $\bar{\lambda}$.
\subsection{Verifying hard case 2}
In this subsection, we discuss how one can detect  hard case 2 if it is the case.  In \cite{wol}, Pong and Wolkowicz  proposed an algorithm based on minimum generalized eigenvalue of a parameterized matrix pencil to solve  QCQP with the two-sided constraint that first
carries out a preprocessing technique to recognize hard case 2, see Section 3.1.1 of \cite{wol} for more details.  According to their result, when $\underline{\lambda} \in I $ and $\lambda^{*}\in [\underline{\lambda},\hat{\lambda})$, QCQP is hard case 2 if and only if $g(x(\underline{\lambda}))\leq 0$ where
$x(\underline{\lambda})=-(A+\underline{\lambda}B)^{\dagger}(a+\underline{\lambda}b)$. Similarly,
 when  $\bar{\lambda}\in I$ and $\lambda^{*}\in (\hat{\lambda},\bar{\lambda}]$, QCQP is hard case 2 if and only if $g(x(\bar{\lambda}))\geq 0$ where $x(\bar{\lambda})=-(A+\bar{\lambda}B)^{\dagger}(a+\bar{\lambda}b)$. Unfortunately, in general  these assertions are not true,  QCQP may be hard case 2 but $g(x(\underline{\lambda}))> 0$ or $ g(x(\bar{\lambda}))<0$. A problem of this type is given in the following example.
 \begin{example}
 Consider a QCQP with
\begin{align*}
A=\begin{bmatrix}
-1&0\\0&1
\end{bmatrix},\,
B=\begin{bmatrix}
2&0\\0&-1
\end{bmatrix},\,
a=\begin{bmatrix}
-25\\-\frac{33}{2}
\end{bmatrix},\,
b=\begin{bmatrix}
50\\25
\end{bmatrix}, \beta=0,
\end{align*}
and set $\hat{\lambda}=\frac{3}{4}$. It is easy to see that  $\underline{\lambda}=\frac{1}{2}$, $\bar{\lambda}=1$, $\phi(\hat{\lambda})<0$ and $(a+\underline{\lambda}b)\in \text{Range}(A+\underline{\lambda}B)$,  and so $\lambda^{*}\in [\underline{\lambda},\hat{\lambda})$. This problem is hard case 2 since $x^{*}=[-25+\sqrt{457}, 8]^T$ with $\lambda^{*}=\underline{\lambda}$ satisfy the  optimality conditions
\begin{align*}
&(A+\underline{\lambda} B)x^{*}=-(a+\underline{\lambda}b),\\
&g(x^{*})=0.
\end{align*}
On the other hand, we have  $x(\underline{\lambda})=[0,8]^T$ and
$g(x(\underline{\lambda}))=336>0$.
 \end{example}
In what follows, we fill this gap and describe an approach to recognize hard case 2.  We  use the following result.
 \begin{proposition}[\cite{M16}]\label{pcp}
    If $\bar{\lambda}$ is finite, then $v^TBv<0$ for all nonzero  $v\in \text{Null}(A+\bar{\lambda}B) $. If $\underline{\lambda}$ is finite, then $v^TBv>0$ for all nonzero $v\in \text{Null}(A+\underline{\lambda}B) $.
\end{proposition}
Consider the following two cases.\\
{\bf Case 1:}$\bar{\lambda}\in I$ and $\lambda^{*}\in (\hat{\lambda},\bar{\lambda}]$.\\
First notice that since $\bar{\lambda}\in I$, then  $(a+\bar{\lambda}b)\in \text{Range}(A+\bar{\lambda}B)$. Hence, system
\begin{align}\label{sys111}
(A+\bar{\lambda}B)x=-(a+\bar{\lambda}b),
\end{align}
 is consistent. Assume that the set $\{v_1,v_2,...,v_r\}$ is an orthonormal basis of $\text{Null}(A+\bar{\lambda}B)$. Set $Z= [v_1, v_2,\cdots, v_r] $, then any solution of system (\ref{sys111}) has the form $x=x(\bar{\lambda})+Zy$ where $x(\bar{\lambda})=-(A+\bar{\lambda}B)^{\dagger}(a+\bar{\lambda}b)$ and $y\in \mathbb{R}^r$. Now consider the following maximization problem:
 \begin{align}\label{ma1}
 p^{*}:=\max_{y\in \mathbb{R}^r} \quad & g(x(\bar{\lambda})+Zy)= y^TZ^TBZy+2y^TZ^T(Bx(\bar{\lambda})+b)+g(x(\bar{\lambda})).
 \end{align}
 By Proposition 2, $Z^TBZ\prec 0$ and hence the optimal solution of (\ref{ma1}), $y^{*}$, is the unique solution of the positive definite system
 \begin{align}
 -(Z^TBZ)y^{*}=Z^T(Bx(\bar{\lambda})+b),
 \end{align}
and  $ p^{*}=g(x(\bar{\lambda})+Zy^{*})$.
In the sequel,  we show that $\lambda^{*}=\bar{\lambda}$ if and only if $p^{*}\geq 0$. To see this, suppose that $p^{*}\geq 0$. Set $x^{*}=x(\bar{\lambda})+Zy^{*}$. Next consider the following quadratic equation of variable $\alpha$:
\begin{align*}
 v^TBv \alpha^{2}+2\alpha v^T(Bx^{*}+b)+p^{*}=0,
\end{align*}
where $v\in \text{Null}(A+\bar{\lambda}B)$. Due to the facts that $v^TBv<0$ and $p^{*}\geq 0$, the above equation has a root $\alpha$. Now it is easy to see that $x^{*}:=x^{*}+\alpha v$ with $\lambda^{*}=\bar{\lambda}$ satisfies the optimality conditions of QCQP:
\begin{align*}
&(A+\bar{\lambda}B)x^{*}=-(a+\bar{\lambda}b),\\
&g(x^{*})=0.
\end{align*}
Next suppose that $\lambda^{*}=\bar{\lambda}$ and $x^{*}$ is a global optimal solution of QCQP. Since $\bar{\lambda}>0$, it follows from optimality conditions that $g(x^{*})=0$ and thus $p^{*}\geq 0$.\\
{\bf Case 2:} $\underline{\lambda}\in I$ and $\lambda^{*}\in [\underline{\lambda},\hat{\lambda})$.\\
This is similar to Case 1.
First notice that since $\underline{\lambda}\in I$, then $(a+\underline{\lambda}b)\in \text{Range}(A+\underline{\lambda}B)$. Hence, system
\begin{align}\label{sys1}
(A+\underline{\lambda}B)x=-(a+\underline{\lambda}b),
\end{align}
 is consistent. Assume that the set $\{v_1,v_2,...,v_r\}$ is an orthonormal basis of $\text{Null}(A+\underline{\lambda}B)$. Set $Z= [v_1, v_2,\cdots, v_r] $, then any solution of system (\ref{sys1}) has the form $x=x(\underline{\lambda})+Zy$ where $x(\underline{\lambda})=-(A+\underline{\lambda}B)^{\dagger}(a+\underline{\lambda}b)$ and $y\in \mathbb{R}^r$. Now consider the following minimization problem:
 \begin{align}\label{ma}
 p^{*}:=\min_{y\in \mathbb{R}^r} \quad & g(x(\underline{\lambda})+Zy)= y^TZ^TBZy+2y^TZ^T(Bx(\underline{\lambda})+b)+g(x(\underline{\lambda})).
 \end{align}
 By Proposition 2, $Z^TBZ\succ  0$ and hence the optimal solution of (\ref{ma}), $y^{*}$, is the unique solution of the positive definite system
 \begin{align}
 (Z^TBZ)y^{*}=-Z^T(Bx(\underline{\lambda})+b),
 \end{align}
and  $ p^{*}=g(x(\underline{\lambda})+Zy^{*})$.
In the sequel,  we show that $\lambda^{*}=\underline{\lambda}$ if and only if $p^{*}\leq 0$. To see this, suppose that $p^{*}\leq 0$. Set $x^{*}=x(\underline{\lambda})+Zy^{*}$. Next consider the following quadratic equation of variable $\alpha$:
\begin{align*}
 v^TBv \alpha^{2}+2\alpha v^T(Bx^{*}+b)+p^{*}=0,
\end{align*}
where $v\in \text{Null}(A+\underline{\lambda}B)$. Due to the facts that $v^TBv>0$ and $p^{*}\leq  0$, the above equation has a root $\alpha$. Now it is easy to see that $x^{*}:=x^{*}+\alpha v$ with $\lambda^{*}=\underline{\lambda}$ satisfies the optimality conditions of QCQP:
\begin{align*}
&(A+\underline{\lambda}B)x^{*}=-(a+\underline{\lambda}b),\\
&g(x^{*})=0.
\end{align*}
Next suppose that $\lambda^{*}=\underline{\lambda}$ and $x^{*}$ is a global optimal solution of QCQP. Since $\bar{\lambda}>0$, it follows from optimality conditions that $g(x^{*})=0$ and thus $p^{*}\leq  0$.\\
\begin{example}
Consider the
same problem in Example 1. We have $\underline{\lambda}\in I$ and $\lambda^{*}\in [\underline{\lambda},\hat{\lambda})$. Moreover,  $Z=[1,0]^T$  is an orthonormal basis of $\text{Null}(A+\underline{\lambda}B)$ and $x(\underline{\lambda})=[0,8]^T$.
It is easy to see that the optimal solution of problem (\ref{ma}), $y^{*}=-25$ and so  $p^{*}=-914<0$. Therefore, QCQP is a hard case 2 instance and $x^{*}=[-25+\sqrt{457}, 8]^T$ with $\lambda^{*}=\underline{\lambda}$ satisfy the  optimality conditions.
\end{example}
\subsection{Solving the nonlinear equation $\phi(\lambda)=0$ }
When  QCQP is not  hard case 2, $\lambda^{*}$ is the unique solution of the equation $\phi(\lambda)=0$ on the underlying interval containing the optimal Lagrange multiplier. If  $B$  is positive  semidefinite, it can be shown that the function $\phi(\lambda)$ is convex  on $I$ and thus a safeguarded version of Newton's method is a reasonable  choice for the solution of this equation. This is the approach used by Mor\'{e} and Sorenson \cite{moretrs}  for solving the TRS which at each iteration makes use of the Cholesky decomposition and hence is not proper for large-scale problems. If $B$ is indefinite, then $\phi(\lambda)$ may not be convex or concave and hence there is no guarantee that Newton's method will be convergent \cite{M16}. Here,
we propose an algorithm for the solution of  $\phi(\lambda)=0$ which is indeed the bisection method occupied with two techniques to accelerate its convergence \cite{wol,salahi}.
\\ \\
{\bf Algorithm for solving  equation $\phi(\lambda)=0$ }\\
 Iterate until a termination criterion is met:
\begin{enumerate}
\item Set $\lambda_{new}$ to the midpoint of the interval containing $\lambda^{*}$.
\item If the points at which $\phi(\lambda)$ has positive and negative values are available,
do inverse linear  interpolation, update $\lambda_{new}$ if inverse linear interpolation is successful.
\item At $\lambda_{new}$, take a step to the boundary if points at which $\phi(\lambda)$ has positive and negative values are known.
\end{enumerate}
{\bf End loop.}\\ \\
In  what follows, we explain two  techniques used in the algorithm.
\subsection*{\bf Inverse Interpolation}
  Let $[\lambda_b, \, \lambda_g]$ be the current interval containing $\lambda^{*}$. Moreover, suppose the relation  (\ref{12})  dose not hold. Then, by Lemma 2, $\phi(\lambda)$ is strictly decreasing  and hence, we can consider its inverse function. We  approximate the inverse  function $\lambda(\phi)$ by a linear function.  Then we set $\lambda_{new}=\lambda(0)$ if $\lambda(0)\in [\lambda_b, \, \lambda_g]$.
\subsection*{\bf Primal step to the boundary}
Let $x^{*}$ denote the optimal solution of QCQP. It can be shown that  as the algorithm proceeds,  $\lambda_{new}$ converges to the $\lambda^{*}$, the optimal Lagrange multiplier and hence, the sequence $x(\lambda)$ also will be  convergent to $x^{*}$. Now, suppose that there exist  values $\lambda_b$ and $\lambda_g$ with $\phi(\lambda_b)>0$ and $\phi(\lambda_g)<0$.  Then, we can take an inexpensive primal step to the point $\alpha x(\lambda_b)+(1-\alpha) x(\lambda_g)$ on the  boundary by choosing a  suitable step length $\alpha$. This  likely improves the objective value. We note that the resulting sequence is also convergent to $x^{*}$.
\subsection*{\bf Computing $\phi(\lambda)$}
At each iteration of the algorithm, $\phi(\lambda)$ is computed by applying  the conjugate gradient algorithm to  the positive definite system of linear equations
\begin{align}\label{ss}
(A+\lambda B)x(\lambda)=-(a+\lambda b).
\end{align}
The conjugate-gradient  method is one of the most widely used iterative
methods for solving symmetric positive-definite linear equations. In hard case 1 when $\lambda^{*}$ is near to $\bar{\lambda}$ or $\underline{\lambda}$,  system  (\ref{ss}) may become  ill-conditioned. In the following two theorems (Theorems \ref{TH2} and \ref{TH3}), we show that, to overcome near singularity, one can solve an alternative well-conditioned  positive definite system that has the same solution as (\ref{ss}). Before stating that,
we need the following proposition and lemma.
 \begin{proposition}\label{pppe1}
 Suppose that Assumption 2 holds. Then $(\lambda,v)$ is an eigenpair  of the pencil $(A,-B)$ if and only if $(-\frac{1}{\lambda-\hat{\lambda}},(A+\hat{\lambda}B)^{\frac{1}{2}}v)$ is an eigenpair of matrix
  $(A+\hat{\lambda}B)^{-\frac{1}{2}}B(A+\hat{\lambda}B)^{-\frac{1}{2}}$.
 \end{proposition}
\begin{proof}
$(\lambda,v)$ is an eigenpair of pencil $(A,-B)$ if and only if $(A+\lambda B)v=0$,  implying that $(A+\hat{\lambda}B+(\lambda-\hat{\lambda})B)v=0$. Since $(A+\hat{\lambda}B)\succ 0$, it follows that $(I+(\lambda-\hat{\lambda})(A+\hat{\lambda}B)^{-\frac{1}{2}}B(A+\hat{\lambda}B)^{-\frac{1}{2}})(A+\hat{\lambda}B)^{\frac{1}{2}}v=0$. Note that $\lambda\not =\hat{\lambda}$ because $\det(A+\hat{\lambda}B)\not =0$. This implies that $((A+\hat{\lambda}B)^{-\frac{1}{2}}B(A+\hat{\lambda}B)^{-\frac{1}{2}}+\frac{1}{\lambda-\hat{\lambda}}I)(A+\hat{\lambda}B)^{\frac{1}{2}}v=0$ which completes the proof. \hspace{9.3 cm} $\Box$
\end{proof}
\begin{lemma}\label{lll}
Suppose that $\lambda\not =\hat{\lambda}$. Then  $(a+\lambda b)\in\text{Range}(A+\lambda B)$ if and only if $(-b+B(A+\hat{\lambda}B)^{-1}(a+\hat{\lambda}b))\in \text{Range}(A+\lambda B)$.
\end{lemma}
\begin{proof}
$(a+\lambda b)\in \text{Range}(A+\lambda B)$ if and only if  there exists $\bar{x}$ such that $(A+\lambda B)\bar{x}=-(a+\lambda b)$. This implies that
$$[(A+\hat{\lambda}B)+(\lambda-\hat{\lambda})B]\bar{x}=-[(a+\hat{\lambda}b)+(\lambda-\hat{\lambda})b].$$
Next,  the change of variable $\bar{x}=\bar{y}-(A+\hat{\lambda}B)^{-1}(a+\hat{\lambda}b)$ gives
\begin{align}
(A+\lambda B)\bar{y}=(\lambda-\hat{\lambda})(-b+B(A+\hat{\lambda}B)^{-1}(a+\hat{\lambda}b)),
\end{align}
implying that $(-b+B(A+\hat{\lambda}B)^{-1}(a+\hat{\lambda}b))\in \text{Range}(A+\lambda B)$.\hspace{3 cm} $\Box$
\end{proof}
\begin{theorem}\label{TH2}
Suppose that $(a+\underline{\lambda}b)\in \text{Range}(A+\underline{\lambda}B)$. Moreover, assume that the set $\{v_1,\cdots,v_r\}$ is a basis of $\text{Null}(A+\underline{\lambda}B)$ that is $(A+\hat{\lambda} B)$-orthogonal, i.e.,  $v_i^T(A+\hat{\lambda}B)v_i=1$ and $v_i^T(A+\hat{\lambda}B)v_j=0$ for all $i\not = j$. Let  $\hat{\lambda}\not =\lambda \in (\underline{\lambda},\bar{\lambda})$,  then $x^{*}$ is the solution of system
\begin{align}\label{s1}
(A+\lambda B)x=-(a+\lambda b),
\end{align}
if and only if $y^{*}=x^{*}+z $ is the  solution of
\begin{align}\label{s2}
[A+\lambda B +\alpha \sum_{i=1}^r(A+\hat{\lambda}B)v_iv_i^T(A+\hat{\lambda}B)]y=(\lambda -\hat{\lambda})(Bz-b)
\end{align}
where $z=(A+\hat{\lambda}B)^{-1}(a+\hat{\lambda}b)$ and $\alpha $ is an arbitrary  positive constant. The same assertion holds when $\underline{\lambda}$ is replaced by $\bar{\lambda}$.
\end{theorem}
\begin{proof}
It is easy to see that $x^{*}$ is the unique  solution of system (\ref{s1}) if and only if $y^{*}=x^{*}+z$ is the solution of
\begin{align}\label{s3}
(A+\lambda B)y=(\lambda-\hat{\lambda})(Bz-b).
\end{align}
To prove the theorem, it is sufficient to show that systems (\ref{s3}) and (\ref{s2}) are equivalent. System (\ref{s2}) can be rewritten as follows:
\begin{align*}
[(A+\hat{\lambda}B)+(\lambda-\hat{\lambda}) B +\alpha \sum_{i=1}^r(A+\hat{\lambda}B)v_iv_i^T(A+\hat{\lambda}B)]y=(\lambda-\hat{\lambda})(Bz-b).
\end{align*}
Since $(A+\hat{\lambda}B)\succ 0$, we have
\begin{multline}\label{s4}
[I+(\lambda-\hat{\lambda})(A+\hat{\lambda}B)^{-\frac{1}{2}}B(A+\hat{\lambda}B)^{-\frac{1}{2}}+\alpha \sum_{i=1}^r(A+\hat{\lambda}B)^{\frac{1}{2}} v_iv_i^T(A+\hat{\lambda}B)^{\frac{1}{2}}](A+\hat{\lambda}B)^{\frac{1}{2}}y\\=(\lambda-\hat{\lambda})(A+\hat{\lambda}B)^{-\frac{1}{2}}(Bz-b).
\end{multline}
Set $M=(A+\hat{\lambda}B)^{-\frac{1}{2}}B(A+\hat{\lambda}B)^{-\frac{1}{2}}$, then by Proposition \ref{pppe1},  $q_i=(A+\hat{\lambda}B)^{\frac{1}{2}}v_i, i=1,...,r$,  are the eigenvectors of $M$ corresponding to the eigenvalue $-\frac{1}{\underline{\lambda}-\hat{\lambda}}$. Note that $q_i^Tq_i=1$ for $i=1,...,r$ and $q_i^Tq_j=0$ for all $i\not =j\in\{1,...,r\}$. Without loss of generality, let  $M=QDQ^T$ be the eigenvalue decomposition of $M$ in which $Q$ contains $q_i=1,...,r$ as its $r$ first columns.
   From (\ref{s4}) we get
\begin{multline}\label{s6}
[I+(\lambda-\hat{\lambda})D+\alpha \sum_{i=1}^rQ^T(A+\hat{\lambda}B)^{\frac{1}{2}}v_iv_i^T(A+\hat{\lambda}B)^{\frac{1}{2}}Q]Q^T(A+\hat{\lambda}B)^{\frac{1}{2}}y=\\(\lambda-\hat{\lambda})Q^T(A+\hat{\lambda}B)^{-\frac{1}{2}}(Bz-b).
\end{multline}
Furthermore,  we have
\begin{align}\label{s7}
Q^T(A+\hat{\lambda}B)^{\frac{1}{2}}v_iv_i^T(A+\hat{\lambda}B)^{\frac{1}{2}}Q=e_ie_i^T,\quad i=1,...,r,
\end{align}
due to the fact that
\begin{align*}
&q_k^T(A+\hat{\lambda}B)^{\frac{1}{2}}v_iv_i^T(A+\hat{\lambda}B)^{\frac{1}{2}}q_j=
\begin{cases}
1&if \,\, k=j=i\in\{1,...,r\},\\
0& otherwise,
\end{cases}
\end{align*}
where $e_i$ is the unit vector  and $q_i,i=1,...,n$, denote $i$'th column of $Q$.
 Therefore, it follows from (\ref{s6}) and (\ref{s7}) that
\begin{align}\label{s8}
(I+(\lambda-\hat{\lambda})D+\alpha \sum_{i=1}^re_ie_i^T)Q^T(A+\hat{\lambda}B)^{\frac{1}{2}}y=(\lambda-\hat{\lambda})Q^T(A+\hat{\lambda}B)^{-\frac{1}{2}}(Bz-b).
\end{align}
Now since $(a+\underline{\lambda}b)\in \text{Range}(A+\underline{\lambda}B)$, by Lemma \ref{lll}, $(Bz-b)\in \text{Range}(A+\underline{\lambda} B)$ and consequently  $v_i^T(Bz-b)=0$ for $ i=1,...,r$. This implies that
$$q_i^T(A+\hat{\lambda}B)^{-\frac{1}{2}}(Bz-b)=0, \quad i=1,...,r.$$
Therefore, the $r$ first components of the right hand side vector in system  (\ref{s8}) are zero. Further let $c_i, i=1,...,r$, denote the $r$ first diagonal elements of matrix $(I+(\lambda-\hat{\lambda})D) $. By Proposition \ref{pppe1}, it is easy to see that $ c_i=\frac{\underline{\lambda}-\lambda}{\underline{\lambda}-\hat{\lambda}}>0,  i=1,...,r$. Since $\alpha>0$, the above discussion proves that  system (\ref{s8})
 is equivalent to the following system:
 \begin{align*}
(I+(\lambda-\hat{\lambda})D)Q^T(A+\hat{\lambda}B)^{\frac{1}{2}}y=(\lambda-\hat{\lambda})Q^T(A+\hat{\lambda}B)^{-\frac{1}{2}}(Bz-b).
\end{align*}
This is also equivalent to
$$(A+\lambda B)y=(\lambda-\hat{\lambda})(Bz-b),$$
which completes the proof. When $\underline{\lambda}$ is replaced by $\bar{\lambda}$, the assertion can be proved in a similar manner.
\hspace{9 cm} $\Box$
\end{proof}
\begin{theorem}\label{TH3}
Suppose that $(a+\underline{\lambda}b)\in \text{Range}(A+\underline{\lambda}B)$. Moreover, assume that the set $\{v_1,\cdots,v_r\}$ is a basis of $\text{Null}(A+\underline{\lambda}B)$ that is $(A+\hat{\lambda} B)$-orthogonal, i.e.,  $v_i^T(A+\hat{\lambda}B)v_i=1$ and $v_i^T(A+\hat{\lambda}B)v_j=0$ for all $i\not = j$.  Then matrix $\tilde{A}=
A+\underline{\lambda }B +\alpha \sum_{i=1}^r(A+\hat{\lambda}B)v_iv_i^T(A+\hat{\lambda}B)
$
is positive definite. The same assertion holds when $\underline{\lambda}$ is replaced by $\bar{\lambda}$.
\end{theorem}
\begin{proof}
Since $A+\underline{\lambda}B\succeq 0$ and $\sum_{i=1}^r(A+\hat{\lambda}B)v_iv_i^T(A+\hat{\lambda}B)\succeq 0$, then $\tilde{A}\succeq 0$. To show that $\tilde{A}\succ 0$ it is sufficient to prove that if $x^T\tilde{A}x=0$ for $x\in \mathbb{R}^n$, then $x=0$. For any $x\in \mathbb{R}^n$, there exist unique $x_1\in \text{Null}(A+\underline{\lambda}B)$ and $x_2\in \text{Range}(A+\underline{\lambda}B)$ such that $x=x_1+x_2$.
We have $$x^T(\tilde{A})x=x_2^T(A+\underline{\lambda}B)x_2+\alpha \sum_{i=1}^r(v_i^T(A+\hat{\lambda}B)x)^2=0.$$
This implies that
\begin{align}\label{en}
 x_2^T(A+\underline{\lambda}B)x_2=0, \quad v_i^T(A+\hat{\lambda}B)x=0, \quad i=1,...,r.
 \end{align}
 Since $(A+\underline{\lambda}B) \succeq 0$, we obtain $ (A+\underline{\lambda}B)x_2=0$, and  hence  $x_2\in \text{Null}(A+\underline{\lambda}B)$. Together with the assumption that $x_2 \in \text{Range}(A+\underline{\lambda} B)$ we obtain $x_2=0$. Thus  $x=x_1$. On the other hand $x_1$ can be written as $x_1=\sum_{j=1}^rc_jv_j$ where $c_j\in \mathbb{R} $ for $j=1,...r$. So it follows from (\ref{en}) that
 $\sum_{j=1}^rc_jv_i^T(A+\hat{\lambda}B)v_j=0$ for $i=1,...,r$, resulting in
  $c_j=0$ for $j=1,...,r$. Therefore $x_1=0$ and consequently $x=0$, which completes the proof.
  When $\underline{\lambda}$ is replaced by $\bar{\lambda}$, the assertion can be proved in a similar manner.
   \hspace{9.5 cm}$\Box$
\end{proof}
\section{Numerical  Experiments}
In this section, on several randomly generated problems of various dimensions, we compare the proposed conjugate gradient-based algorithm (CGB) with Algorithms 1 and 2 from  \cite{novel} and Algorithm 3.2 from \cite{qeig}.
The comparison with Algorithm 3.2 of \cite{qeig} is done for dimension up to 5000 as it needs longer time to solve larger dimensions.
All computations are  performed in MATLAB 8.5.0.197613 on a 1.70 GHz laptop with 4 GB of RAM.
Throughout the paper, we have assumed that there exists $\hat{\lambda}$ such that $A+\hat{\lambda} B \succ 0$. In practice $\hat{\lambda}$ is usually unknown in advance, but since all four  algorithms require it for initialization, we assume that $\hat{\lambda}$ is known and skip its computation. However, the algorithms in \cite{hat2,hat1,M16} can be used to find $\hat{\lambda}$. We use the stopping criterion
\begin{align*}
\max\quad  \{ |\phi(\lambda)|, ||(A+\lambda B)x+(a+\lambda b)|| \} <10^{-8},
\end{align*}
or
\begin{align*}
\frac{|\text{high}-\text{low}|}{|\text{high}| + |\text{low}|} <10^{-11}
\end{align*}
for solving the equation $\phi(\lambda)=0$ where $x$ is the best feasible solution of QCQP obtained up to the current iteration and $[\text{low},\text{high}]$ is the  interval containing $\lambda^{*}$ at the current iteration.  We set $\epsilon_1=10^{-10}$, $\epsilon_2=10^{-13}$ and $\epsilon_3=10^{-10}$  in Algorithms 1 and 2 from \cite{novel}.  As in \cite{novel}, to compute $\underline{\lambda}$ and $\bar{\lambda}$, the generalized eigenvalue problem is
solved by \verb|eigifp|\footnote{eigifp is a MATLAB program for computing a few algebraically smallest or largest eigenvalues and their
corresponding eigenvectors of the generalized eigenvalue problem, available from http://www.ms.uky.edu/
~qye/software.html.}. To verify hard case 2, as described in Section 3.1,
 we need to compute an  orthonormal basis of $\text{Null}(A+\underline{\lambda} B)$ or $\text{Null}(A+\bar{\lambda} B)$. Recall that $A+\underline{\lambda}B$ and $A+\bar{\lambda}B$ are singular and positive semidefinite. Thus, a nullspace vector of $A+\underline{\lambda}B$ and $A+\bar{\lambda}B$ can be found
by finding an eigenvector corresponding to the smallest eigenvalue, i.e., 0, which  in our numerical tests, is done by  \verb|eigifp|. We apply the Newton  refinement process in Section
4.1.2 of \cite{qeig} to all four algorithms to improve the accuracy of the solution.
For simplicity, we consider QCQP with nonsingular $B$ including
 positive definite and indefinite $B$. Hence, we can assume without loss of generality that
$b = 0$. Our test problems include both easy and hard case (case 1 and case 2) instances.
\subsection{First class of test problems}
In this class of test problems, we consider  QCQP with positive definite $A$ and indefinite $B$. Thus, we set $\hat{\lambda}=0$.
 We follow Procedure 1, Procedure 2 and Procedure 3 to generate an easy case, hard case 1 and hard case  2 instance, respectively\\
{\bf Procedure 1. (Generating an easy case instance)  }\\
We first generate randomly a   sparse positive definite matrix $A$ and a sparse indefinite matrix $B $ via \verb|A=sprandsym(n,density,1/cond,2)| and
\verb| B = sprandsym(n,density)| where
\verb|cond| refers to condition number. After computing $\bar{\lambda}$,  we set $a = -(A +\lambda B)x_0$, where  $x_0$ is computed via \verb|x0=randn(n,1)/10| and $\lambda$ is chosen uniformly from $(0,\bar{\lambda})$. Next, to avoid generating a QCQP   having an interior solution, we chose $\beta$ randomly  from $(-s ,-\ell)$ where $s=x_c^TBx_c$,  $\ell=x_0^TBx_0$ and  $x_c$ is the solution of linear system $Ax=-a$. From the optimality
conditions  we see that the above construction likely gives an easy case instance.\\
{\bf Procedure 2. (Generating a hard case 1 instance)  }\\
We generate   $A$, $B $, $x_0$ and $\beta$ as the above procedure  but  we set $a = -(A +\bar{\lambda} B)x_0$. From the optimality
conditions  we see that this construction gives a hard  case  1 instance.\\
{\bf Procedure 3. (Generating a hard case 2 instance)  }\\
We  generate   $A$, $B $ and $a$ as Procedure 2 but  we set $\beta=-\ell$.\\

We set $cond=10,100,1000$ and for each dimension and each condition number, we generated 10 instances, and the corresponding numerical results are adjusted in Tables 1 and 2, where we
report   the average
runtime in second (Time)  and accuracy (Accuracy). To measure the accuracy, we have computed the relative
objective function difference as follows:
$$ \frac{q(x^{*})-q(x_{best})}{|q(x_{best})|},$$
where $x^{*}$ is the computed solution by each method and $x_{best }$ is the solution with the
smallest objective value  among the four  algorithms.  We use "Alg1" and  "Alg2" to denote  Algorithms 1 and 2 in \cite{novel}, respectively,  and "Alg3" to denote Algorithm 3.2 in \cite{qeig}.

\begin{table}
\caption{Computational results of the first class of test problems with  density=1e-2}
\noindent
\include{apbinja}
\end{table}

\begin{table}
\caption{Computational results of the first class of test problems with density=1e-4}
\include{newapbin}
\end{table}

\subsection{Second class of test problems}
We first generate randomly a sparse positive definite matrix $C$ and a sparse indefinite matrix $B$ via \verb|C=sprandsym(n,density,1/cond,2)| and
\verb| B = sprandsym(n,density)|. Next, we set $A=C-B$. In this case, obviously, we can set $\hat{\lambda}=1$.  We follow  Procedure 1 to generate an easy case instance but we choose  $\lambda$ uniformly from $(\underline{\lambda}, \bar{\lambda})$ and  if $s>\ell$, we choose $\beta$  randomly from $(-s,-\ell)$ else we choose $\beta $  from $ (-\ell,-s)$ where $s=x_c^TBx_c$ and $s$ is the solution of linear system $Cx=-a$. Procedure 2  and 3 are followed to generate hard case   (case 1 and 2) instances.
The numerical results are adjusted in Tables 3 and 4.

As we see in Tables 1 and 3,  for $90\% $ of the generated instances, CGB algorithm is faster than Algorithm 3.2 from \cite{qeig} while having comparable accuracy. The time difference become significant when we increase the dimension, since  Algorithm 3.2 requires computing an extremal eigenpair of an $(2n +1)  \times (2n +1)$ generalized eigenvalue problem  which is time-consuming when the involved matrices are not highly sparse.

From Tables 1, 2 , 3 and 4,  we observe that CGB algorithm computes more accurate solutions than Algorithms 1 and 2.   Moreover,  except for one dimension, it is always faster than Algorithm 1 and when we increase the dimension, the time difference become significant.  For easy case problems, CGB Algorithm is faster than Algorithm 2 for about 77\% of cases, specifically,  CGB Algorithm is more efficient than Algorithm 2 when condition number is large.   In hard Case 2,  CGB Algorithm is much more efficient than Algorithms 1 and 2 because it first checks for hard case 2, if it is the case, the optimal solution of QCQP is computed via solving  a positive definite system of linear equations.  In hard case 1,  CGB algorithm is slower than Algorithm 2 for about 60 \% of the cases   but still faster than Algorithm 1 because of the extra time it take for verifying hard case 2.

\begin{table}
\caption{Computational results of the second class of test problems with  density=1e-2}
\include{ainbinja}
\end{table}

\begin{table}
\caption{Computational results of the second class of test problems with density=1e-4}
\include{newainbin}
\end{table}
\section{Conclusions}
In this paper, we have considered the problem of minimizing a general quadratic function subject to one general quadratic constraint.  A conjugate gradient-based algorithm is introduced to  solve the problem which is based on solving a sequence of positive definite system of linear equations by conjugate gradient method. Our computational experiments
on several randomly generated test problems show that the proposed method is efficient for large sparse instances since
 the most expensive operations at  each iteration  is
several matrix vector products, that are cheap when the matrix is sparse.

\section{Acknowledgement}The authors would like to thank Iran National Science Foundation (INSF) for supporting this research under project no. 95843381.

\end{document}

%% file: apbinja.tex
\begin{tabular}{@{}c|c|c|c|c|c|c|c|c|c @{}}
\hline
  & &\multicolumn{4}{|c|}{Time(s)}& \multicolumn{4}{|c}{Accuracy} \\
  \hline
 &n & CGB & Alg1  & Alg2 & Alg3 & CGB & Alg1 & Alg2 & Alg3 \\
 \hline
 cond=10&\multicolumn{4}{|c}{}\\
 \hline
 &1000 &  \bf{0.24} &  0.32&    0.27  &     0.60   &  0.0000e+00 & 1.3534e-14& 2.9908e-14   & 2.703592e-16 \\
&2000 &  0.36 &  0.43&    0.36 &     3.85   &  9.5123e-17  & 2.1722e-15& 2.7597e-15   & 4.640596e-17 \\
Easy Case&3000 &   \bf{0.50} &  0.79&    0.52  &    11.56   &  4.4606e-17   & 5.4207e-16& 3.9725e-16   & 7.106022e-17 \\
&4000 &  \bf{0.92} &  1.40&    1.24  &    30.17   &  2.4456e-16  & 1.7972e-15& 2.2517e-15   & 1.711915e-16 \\
&5000 &  \bf{1.14 }&  1.66&    1.25   &   417.53   &  4.0827e-16   & 4.2164e-16& 4.5748e-16   & 2.455677e-16 \\
\hline
cond=100&\multicolumn{4}{|c}{}\\
\hline
&1000 &  0.56&  0.74&    0.58 &     \bf{0.54}  & 2.2119e-16   & 7.7736e-13& 1.3902e-11   &2.089449e-16  \\
&2000 &  \bf{0.87} &  1.16&    0.92  &     3.61   &6.9482e-17   & 1.2703e-13& 5.0994e-13   & 1.123734e-16 \\
Easy Case &3000& \bf{1.98} &  2.84&    2.32  &   12.70& 1.2295e-16   &  6.3770e-14   & 2.6147e-12& 4.2882e-17  \\
&4000 & \bf{1.98}&  2.84&    2.32 &   50.70& 1.6666e-16   &  4.2965e-14   & 4.6357e-13& 5.4903e-17  \\
&5000 &   \bf{4.18}&  6.97&    5.16    &   673.36   &  2.5057e-16   & 1.8041e-14& 5.3498e-14   &1.4510e-17 \\
\hline
cond=1000&\multicolumn{4}{|c}{}\\
\hline
&1000 &   3.75 &  5.33&    4.13  &     \bf{0.56}    &  1.3319e-16   & 1.2508e-11& 1.5758e-09   & 3.415593e-16 \\
&2000 &   4.29 &  6.47&    4.45  &     \bf{3.84}   &  1.5526e-16   & 3.3995e-12& 2.8861e-10   & 1.421543e-16 \\
Easy Case&3000 &   \bf{8.59} & 17.89&    9.56  &    12.17   & 4.3130e-15  & 1.2760e-12& 8.4678e-11   & 2.068219e-16 \\
&4000 &  \bf{10.56} & 18.10&   13.71  &    41.88   &  1.2220e-15   & 5.5054e-13& 1.1153e-11   & 6.020699e-17 \\
&5000 &  \bf{16.36} & 30.95&   21.69  &   371.55   &  1.7654e-16   & 4.5521e-13& 3.9881e-11   & 1.804393e-16 \\
\hline
cond=10&\multicolumn{4}{|c}{}\\
\hline
&1000 &   0.29 &  0.30&    \bf{0.23}  &     0.68   &  2.5820e-16   & 4.6885e-14& 4.7315e-13   & 2.141403e-16 \\
&2000 &   \bf{0.49} &  1.53&    0.72  &     5.33   &  8.9252e-17   & 2.9432e-13& 3.1029e-11   & 9.468005e-17 \\
Hard Case 1&3000 &   1.60 &  1.87&    \bf{1.56}  &    18.63   &  8.8924e-17   & 1.0251e-14& 8.0005e-14   & 8.529213e-17 \\
&4000 &   \bf{1.34} &  1.63&    1.50  &    48.54   &  7.6730e-17   & 4.5916e-15& 5.3948e-15   & 1.863880e-16 \\
&5000 &   \bf{6.09} &  9.09&    6.59  &   307.54   &  1.7215e-16   & 4.9540e-15& 2.8191e-15   & 4.814561e-16 \\
\hline
cond=100&\multicolumn{4}{|c}{}\\
\hline
&1000 &   0.37 &  0.72&    \bf{0.33}  &     0.61   &  3.4527e-16   & 2.1326e-12& 1.1098e-10   & 1.517419e-16 \\
&2000 &   \bf{0.70} &  2.26&    1.07  &     4.40   &  1.4639e-16   & 1.8446e-12& 2.0338e-11   & 1.263138e-16 \\
Hard Case 1&3000 &   3.46 &  4.41&    \bf{3.29}  &    12.75   &  4.1817e-16  & 1.5004e-13& 1.8512e-12   &5.416258e-17 \\
&4000 &   \bf{5.36} &  8.24&    6.14  &    43.24   &  2.8051e-16   & 2.1693e-13& 6.9215e-12   & 1.061026e-16 \\
&5000 &   \bf{9.04} & 13.48&   10.05  &   310.26   &  8.2448e-17   & 6.1508e-14& 1.5047e-13   & 4.103097e-17 \\
\hline
cond=1000&\multicolumn{4}{|c}{}\\
\hline
&1000 &   1.41 &  6.77&    1.38  &     \bf{0.70}   &  1.7472e-16   & 1.0359e-10& 6.6120e-09   & 1.024406e-16 \\
&2000 &   4.52 &  7.23&    4.77  &     \bf{4.04}   &  3.6154e-16   & 3.4160e-12& 2.2689e-10   & 1.746964e-16 \\
Hard Case 1&3000 &  \bf{11.11} & 29.03&   13.83 &  20.93   &  1.5364e-16   & 9.3206e-11& 2.6199e-09   & 1.716415e-16 \\
&4000 &  \bf{23.37} & 47.93&   27.75  &    52.19   &  1.3664e-16   & 1.7287e-12& 1.0004e-10   & 1.511822e-16 \\
&5000 &  \bf{21.09} & 46.43&   28.84  &   402.58   &  6.0663e-16   & 1.5179e-12& 2.0300e-10   & 2.686279e-16 \\
\hline
cond=10&\multicolumn{4}{|c}{}\\
\hline
&1000 &   \bf{0.26} &  0.88&    0.36  &     3.00   &  7.4723e-17   & 4.9359e-13& 4.8652e-11   & 5.471231e-16 \\
&2000 &   \bf{0.36} &  1.58&    0.73  &    16.70   &  0.0000e+00  & 1.6001e-13& 3.8674e-13   & 3.896115e-16 \\
Hard Case 2&3000 &   \bf{0.77} &  7.40&    2.45  &   173.19   &  3.4482e-16   & 2.4531e-13& 7.8146e-11   & 3.030414e-16 \\
&4000 &   \bf{1.44} & 12.18&    3.98  &   852.45   &  4.5639e-17   & 8.9613e-14& 1.3460e-12   & 2.422850e-16 \\

&5000 &   \bf{2.03} & 30.51&    5.31  &   1463.27   &  3.0985e-16   & 2.0535e-13& 3.5345e-12   & 1.831296e-16 \\
\hline
cond=100&\multicolumn{4}{|c}{}\\
\hline
&1000 &   \bf{0.62} &  7.21&    2.12  &     4.23   &  2.1081e-16   & 2.7349e-11& 4.0860e-09   & 6.574916e-16 \\
&2000 &   \bf{1.01} &  7.94&    3.01  &    15.10   &  1.3175e-16   & 3.5787e-12& 5.1980e-10   & 8.660221e-17 \\
Hard Case 2&3000 &   \bf{2.71} & 46.79&    9.41  &   236.70   &  1.0358e-16   & 1.3325e-11& 2.5842e-09   & 9.794199e-17 \\
&4000 &   \bf{4.29} &118.69&   20.47  &   495.63   &  6.4502e-17   & 7.7936e-12& 3.0469e-09   & 6.557389e-17 \\
& 5000 &   \bf{8.22} &209.70&   52.54  &   697.56   &  1.2621e-16   & 6.3835e-12& 7.1765e-10   & 6.557389e-17 \\

\hline
cond=1000&\multicolumn{4}{|c}{}\\
\hline
&1000 &   \bf{2.16} & 14.14&    4.75  &     3.24   &  3.6635e-16   & 3.9804e-09& 5.2611e-08   & 1.920469e-16 \\
&2000 &   \bf{3.49} & 31.50&    7.13  &    29.45   &  4.3374e-17  & 6.5600e-09& 1.9302e-08   & 1.173405e-16 \\
Hard Case 2&3000 &   \bf{9.67} &122.63&   27.69  &    79.25   &  4.2925e-17  & 3.2570e-09& 3.9270e-08   & 2.720614e-16 \\
&4000 &  \bf{12.81} &167.30&   43.43  &   599.55   &  1.2716e-16   & 1.9239e-08& 1.0997e-08   & 1.315645e-16 \\
&5000 &  \bf{17.62} &245.88&   83.07  &   1151.76   &  2.1056e-16   & 2.6294e-08& 3.0084e-08   & 1.324996e-16 \\
\hline
\end{tabular}

%% file: newapbin.tex
\begin{tabular}{c|c|c|c|c|c|c|c}
\hline
 & &\multicolumn{3}{|c|}{Time(s)}& \multicolumn{3}{|c}{Accuracy} \\
\hline
 &n & CGB & Alg1  & Alg2 & CGB & Alg1 & Alg2 \\
 \hline
 cond=10&\multicolumn{4}{|c}{}\\
 \hline
&10000 &   0.47 &  0.62&    \bf{0.43}  &   1.1725e-15  &  2.9305e-15   & 1.9469e-15\\
&20000 &   1.10 &  1.65&    1.10  &   1.6881e-15  &  7.8945e-16   & 6.9847e-16\\
Easy Case &30000 &   1.43 &  2.29&    \bf{1.34}  &   8.5874e-16  &  5.0097e-16   & 4.8177e-16\\
&40000 &   \bf{2.78} &  4.07&    2.82  &   3.9322e-16  &  8.6336e-16   & 6.7728e-16\\
&50000 &   \bf{3.53} &  5.74&    3.82  &  2.1866e-16 &  1.5472e-15   & 1.6432e-14\\
 \hline
  cond=100&\multicolumn{4}{|c}{}\\
 \hline
 &10000 &   \bf{2.79} &  3.69&    2.84  &   0.0000e+00  &  5.1216e-13   & 9.4987e-12\\
 &20000 &   \bf{4.23} &  6.40&    4.73  &   0.0000e+00  &  5.3774e-14   & 1.7939e-13\\
 Easy Case &30000 &   \bf{4.43} &  6.05&    4.52  &   0.0000e+00  &  2.2114e-14   & 7.4127e-14\\
 &40000 &   \bf{8.26} & 12.96&    8.76  &   1.6904e-16  &  2.7534e-14   & 9.8461e-14\\
&50000 &   \bf{8.70} & 11.88&    9.25  &   2.9568e-16  &  9.3602e-15   & 2.6464e-15\\
 \hline
 cond=1000&\multicolumn{4}{|c}{}\\
 \hline
       & 10000 &   \bf{8.48} & 13.33&    9.09  &   0.0000e+00  &  6.7834e-12   & 7.8038e-10\\
        &20000 &  \bf{14.55} & 23.77&   16.44  &   0.0000e+00  &  7.7513e-13   & 2.9847e-11\\
       Easy Case  &30000 &  \bf{23.90} & 40.75&   27.00  &   0.0000e+00  &  4.3943e-13   & 1.2151e-11\\
&40000 &  \bf{44.37} & 71.48&   50.39  &   0.0000e+00  &  3.7856e-13   & 5.2588e-12\\
         &50000 &  \bf{54.38} & 84.49&   63.40  &    0.0000e+00   &  2.6751e-13   & 4.2168e-12  \\
\hline
cond=10&\multicolumn{4}{|c}{}\\
\hline
&10000 &   0.88 &  0.96&    \bf{0.64}  &   5.5224e-17  &  2.2499e-14   & 2.7961e-13\\
&20000 &   1.60 &  1.58&    \bf{1.08}  &   3.6226e-16  &  5.7196e-15   & 1.0666e-14\\
Hard Case 1&30000 &   3.61 &  3.90&    \bf{2.88}  &   4.2608e-16 &  1.6371e-14   & 2.1939e-14\\
&40000 &   4.34 &  3.87&    \bf{2.40}  &   1.7872e-15  &  2.1961e-15   & 2.1956e-16\\
&50000 &   6.73 &  6.72&    \bf{5.20}  &   1.1921e-15  &  1.6576e-14   & 6.3980e-15\\
\hline
cond=100&\multicolumn{4}{|c}{}\\
\hline
&10000 &   2.91 &  4.10&    \bf{2.65}  &   0.0000e+00 &  1.0370e-12   & 3.1089e-11\\
&20000 &   4.40 &  5.36&    \bf{3.30}  &   0.0000e+00  &  1.0955e-13   & 1.4446e-11\\
Hard Case 1&30000 &  12.51 & 14.60&    \bf{7.99}  &   0.0000e+00  &  2.4296e-13   & 8.2524e-13\\
&40000 &  10.66 & 10.68&    \bf{7.37}  &   3.2416e-15  &  1.7511e-14   & 1.8957e-14\\
&50000 &  19.62 & 16.69&   \bf{11.66}  &   5.8187e-15  &  9.0692e-15   & 3.5658e-15\\
\hline
cond=1000&\multicolumn{4}{|c}{}\\
\hline
&10000 &  19.04 & 20.76&   \bf{14.89}  &   0.0000e+00  &  7.8027e-12   & 1.6214e-09\\
&20000 &  28.85 & 45.77&   \bf{27.56}  &   0.0000e+00  &  1.5992e-12   & 1.1197e-10\\
Hard Case 1&30000 &  \bf{41.50} & 64.18&   42.81  &   0.0000e+00  &  9.0035e-13   & 2.7761e-11\\
&40000 &  \bf{60.11} &106.70&   61.81  &   0.0000e+00  &  9.8912e-13   & 2.2214e-11\\
&50000 &  58.34 & 93.55&   \bf{58.01}  &   0.0000e+00  &  4.6131e-13   & 4.4259e-12\\
\hline
 cond=10&\multicolumn{4}{|c}{}\\
 \hline
            &10000 &   \bf{0.41} &    4.25  &    1.22& 0.0000e+00   &  2.3268e-13   & 5.8761e-11    \\
            &20000 &   \bf{0.76}  &    3.50  &   1.04& 0.0000e+00  &  5.8964e-14   & 7.2553e-14    \\
Hard Case 2&30000 &   \bf{1.74} &   25.12  &    9.12& 0.0000e+00   &  6.8237e-13   & 5.6535e-12    \\
            &40000 &   \bf{2.06} &   14.36  &    6.11& 0.0000e+00   &  1.3294e-13   & 3.8901e-13    \\
             &50000 &   \bf{2.78} &   14.73  &    5.20& 0.0000e+00  &  1.0006e-13   & 7.4351e-14    \\
\hline
cond=100&\multicolumn{4}{|c}{}\\
\hline
           &10000 &   \bf{1.78}&   15.16 &    3.42& 0.0000e+00   &  5.5265e-10   & 5.1118e-09    \\
           &20000 &   \bf{1.84} &   25.53  &    5.79& 0.0000e+00  &  6.0442e-12   & 9.0442e-10    \\
Hard Case 2&30000 &   \bf{2.61} &   30.48  &    7.78& 0.0000e+00   &  8.1750e-13   & 1.6605e-11    \\
            &40000 &   \bf{4.39} &   91.63  &   16.24& 0.0000e+00   &  4.7839e-12   & 1.4334e-10    \\
            &50000 &   \bf{7.59} &  136.59  &   27.00& 0.0000e+00   &  4.5553e-11   & 1.3001e-10    \\
\hline
cond=1000&\multicolumn{4}{|c}{}\\
\hline
            &10000 &   \bf{4.70} &   38.58  &    9.19& 0.0000e+00   &  1.5298e-08   & 1.2372e-07    \\
            &20000 &   \bf{7.41} &   92.02  &   20.77& 0.0000e+00   &  6.0570e-09   & 2.0386e-08    \\
Hard Case 2&30000 &  \bf{34.75} &293.60&   79.92  &   0.0000e+00  &  9.4080e-10   & 9.4997e-10\\
            &40000 &  \bf{57.08} &344.94&   70.91  &  0.0000e+00  &  1.5930e-10   & 3.6650e-09\\
           & 50000 &  \bf{53.14} &631.03&  138.77  &   0.0000e+00  &  1.0949e-10   & 6.6408e-10\\
\hline

\end{tabular}

%% file: ainbinja.tex
\begin{tabular}{c|c|c|c|c|c|c|c|c|c}
\hline
  & &\multicolumn{4}{|c|}{Time(s)}& \multicolumn{4}{|c}{Accuracy} \\
  \hline
 &n & CGB & Alg1  & Alg2 & Alg3 & CGB & Alg1 & Alg2 & Alg3 \\
 \hline
 cond=10&\multicolumn{4}{|c}{}\\
 \hline
&1000 &   \bf{0.19} &  0.25&    0.22  &     0.61   &  1.8361e-16   & 2.3145e-14& 1.4939e-13   & 1.806630e-16 \\
&2000 &   \bf{0.35} &  0.48&    0.40  &     4.04   &  1.1614e-16   & 1.0821e-14& 2.3359e-14   & 1.129161e-16 \\
Easy Case &
3000&\bf{0.60} &  0.75&    1.06  &    12.71   &  1.6823e-16   & 1.8640e-15& 2.6605e-15   & 7.371001e-17\\
&4000 &   \bf{1.02} &  1.39&    1.75  &    28.59   &  7.3294e-17   & 1.4715e-15& 3.4793e-13   & 1.517022e-16 \\&5000 &   3.02 &  3.77&    \bf{2.34}  &   325.90   &  6.2226e-16   & 2.7151e-16& 4.2262e-16   & 6.980376e-17 \\
 \hline
 cond=100&\multicolumn{4}{|c}{}\\
 \hline
 &1000 &   \bf{0.48} &  1.02&    0.71  &     0.62   &  5.2875e-16   & 2.6805e-12& 4.0940e-11   & 1.715593e-15 \\
&2000 &    \bf{0.91} &  1.51&    1.10  &     3.78   & 3.0586e-16  & 2.2825e-13& 1.5995e-12   & 0 \\
Easy Case&3000 &    \bf{1.92} &  2.55&    2.00  &    13.62   &  2.3203e-15   & 2.7241e-13& 1.1253e-12   & 1.952349e-16 \\
&4000 &   \bf{2.87} &  4.06&    3.20  &    36.20   &  1.0388e-15   & 2.5531e-13& 1.2496e-12   &6.361166e-16 \\
&5000 &   \bf{4.14} &  5.14&    4.32  &   234.46   &  1.1003e-15   & 1.0004e-13& 3.9197e-13   & 1.497476e-16 \\
\hline
 cond=1000&\multicolumn{4}{|c}{}\\
 \hline
&1000 &   1.85 &  3.25&    2.02  &     \bf{0.55}   &  2.3202e-16   & 6.2333e-11& 1.3809e-04   & 3.307807e-16 \\
&2000 &   5.42 &  8.53&    6.08  &     \bf{4.80}   & 5.8536e-17   & 2.1502e-12& 2.3546e-10   & 3.576414e-16 \\
Easy Case&3000 &  \bf{12.98} & 19.32&   13.94  &    14.35   &  8.4864e-17  & 1.3869e-12& 4.7013e-11   & 1.324596e-16 \\
&4000 & \bf{10.09} & 22.45&   11.43  &    38.59   &  1.4693e-16   & 7.1405e-13& 8.7914e-11   & 4.308794e-17\\
&5000 &  \bf{19.00} & 30.42&   22.90  &   385.47   &  1.0816e-16   & 3.9543e-13& 6.7996e-12   & 1.596002e-16 \\
\hline
cond=10&\multicolumn{4}{|c}{}\\
 \hline
&1000 &   0.44 &  0.47&    \bf{0.34}  &     0.80   &  7.6776e-17   & 3.4128e-14& 1.3917e-13   & 2.167390e-16 \\
&2000 &   0.65 &  0.88&    \bf{0.60}  &     5.25   &  3.2993e-17   & 7.8537e-14& 3.1866e-12   & 1.040181e-16 \\
Hard Case 1&3000 &   1.42 &  1.45&    \bf{1.13}  &    15.14   &  9.2671e-17   & 1.2647e-14& 2.0693e-09   & 6.589870e-17 \\
&4000 &   1.70 &  1.22&    \bf{1.02}  &    35.27   &  7.2135e-17   & 1.7775e-15& 2.4309e-15   & 8.207300e-17 \\
&5000 &   6.44 &  6.86&    \bf{4.31}  &   401.36   &  2.2950e-16   & 1.5855e-14& 2.0637e-13   & 2.502850e-17 \\
\hline
cond=100&\multicolumn{4}{|c}{}\\
 \hline
&1000 &   0.75 &  0.99&    0.65  &     \bf{0.57}   &  1.2815e-16   & 1.3903e-12& 1.3202e-10   & 4.722266e-17 \\
&2000 &   1.33 &  1.95&    \bf{1.16}  &     4.42   &  5.7397e-17   & 5.4877e-13& 1.6352e-11   & 1.161957e-16 \\
Hard Case 1&3000 &   4.05 &  3.24&    \bf{2.72}  &    13.22   &  7.5077e-17   & 1.1536e-13& 2.6095e-12   & 2.778583e-17 \\
&4000 &   5.65 & 11.97&    \bf{5.54}  &    40.99   &  9.2310e-17   & 7.0747e-13& 5.3772e-11   & 4.250603e-17 \\
&5000 &  11.21 & 10.75&    \bf{7.45}  &   332.81   &  2.0919e-16   & 3.5848e-14& 6.5684e-13   &  2.048992e-17 \\
\hline
 cond=1000&\multicolumn{4}{|c}{}\\
 \hline
&1000 &   2.34 &  4.41&    1.99  &     \bf{0.78}   &  4.8791e-17   & 1.0518e-11& 1.6066e-06   & 1.053864e-16 \\
&2000 &   4.40 & 10.19&    4.57  &    \bf{4.36}   &  1.1762e-16   & 4.4629e-12& 4.1546e-10   & 4.459648e-17 \\
Hard Case 1&3000 &  15.41 & 28.43&   \bf{12.78}  &    14.62   &  1.1966e-13   & 3.0246e-12& 2.3939e-10   & 8.079233e-17\\
&4000 &  13.81 & 33.07&   \bf{12.23}  &    27.22   &  7.3406e-15   & 2.0796e-12& 7.9221e-10   & 4.955689e-17 \\
&5000 &  \bf{34.44} &151.09&   45.20  &   263.30   &  3.1287e-15   & 1.2173e-09& 1.4141e-08   & 4.955689e-17 \\
\hline
 cond=10&\multicolumn{4}{|c}{}\\
\hline
&1000 &   \bf{0.33} &  0.86&    0.36  &     3.24   &  0.0000e+00  & 3.8490e-13& 1.8209e-12   & 6.857462e-16 \\
&2000 &   \bf{0.99} &  8.64&    2.41  &    85.23   &  0.0000e+00   & 6.4555e-13& 1.8722e-11   & 3.651448e-16 \\
Hard Case 2& 3000 &   \bf{1.10} &  4.00&    1.83  &    65.52   &  0.0000e+00   & 7.5183e-14& 3.9460e-13   & 6.287905e-17 \\
&4000 &   \bf{1.88} & 16.77&    3.24  &   487.32   &  3.4636e-16   & 1.3285e-13& 4.2813e-12   & 0.0000e+00 \\
&5000 &   \bf{2.45} & 20.90&    6.64  &   909.38   &  5.4219e-16   & 6.9743e-13& 3.0042e-11   & 0.0000e+00 \\
\hline
 cond=100&\multicolumn{4}{|c}{}\\
\hline
&1000 &   \bf{0.53} &  2.39&    1.05  &     1.32   &  2.6963e-16   & 4.9522e-12& 2.9811e-10   & 4.770832e-16 \\
&2000 &   \bf{1.00} &  9.40&    2.90  &    28.23   &  4.0417e-16   & 6.6325e-12& 5.2118e-10   & 1.133704e-16 \\
Hard Case 2& 3000 &   \bf{2.87} & 58.94&    9.74  &   107.78   &  2.9495e-16   & 5.7992e-12& 9.2886e-10   & 0.0000e+00 \\
&4000 &   \bf{4.92} & 31.57&    8.35  &   147.59   &  2.9235e-16   & 1.0990e-12& 1.7460e-10   & 0.0000e+00\\
&5000 &   \bf{7.96} &160.04&   48.60  &   1308.83   &  0.0000e+00   & 2.9739e-12& 4.4620e-10   & 2.581425e-16 \\

\hline
 cond=1000&\multicolumn{4}{|c}{}\\
\hline
&1000 &   \bf{2.68} & 14.00&    6.01  &     4.06   &  2.9054e-16   & 8.7283e-08& 1.5660e-07   & 0 \\
&2000 &   \bf{5.68} & 31.84&   12.25  &    18.86   &  2.5527e-16   & 1.2073e-08& 8.0504e-08   & 6.035149e-17 \\
Hard Case 2&3000 &  \bf{12.38} & 94.31&   29.17  &    53.89   &  2.2569e-16   & 5.4429e-09& 1.0046e-07   & 1.214486e-16 \\
&4000 &  \bf{10.80} &152.22&   29.82  &   174.29   &  6.0252e-17   & 5.2876e-09& 2.9837e-08   & 5.656610e-17 \\
&5000 &  \bf{48.30} &276.64&   70.15  &   449.92   &  2.0865e-16   & 1.8989e-08& 3.4265e-08   & 2.959448e-16 \\
\hline
\end{tabular}

%% file: newainbin.tex
\begin{tabular}{c|c|c|c|c|c|c|c}
\hline
 & &\multicolumn{3}{|c|}{Time(s)}& \multicolumn{3}{|c}{Accuracy} \\
\hline
 &n & CGB & Alg1  & Alg2 & CGB & Alg1 & Alg2 \\
 \hline
cond=10&\multicolumn{4}{|c}{}\\
\hline
&10000 &   0.71 &  0.84&    0.71  &   2.9910e-17   & 2.5167e-15& 6.1729e-15   \\
&20000 &   0.90 &  1.47&    \bf{0.87}  &   2.9690e-16   & 1.0642e-15& 2.6536e-15   \\
Easy Case &30000 &   \bf{1.83} &  2.72&    2.27  &   3.6312e-16   & 1.5595e-15& 2.8782e-15   \\
&40000 &   \bf{2.96} &  4.45&    3.10  &   3.0619e-16   & 1.0188e-15& 6.8731e-16   \\
&50000 &   \bf{5.74} &  8.34&    6.34  &   3.5265e-16   & 9.1985e-16& 1.9643e-15   \\
\hline
cond=100&\multicolumn{4}{|c}{}\\
\hline
&10000 &   1.66 &  2.16&    \bf{1.62}  &   0.0000e+00   & 5.1889e-13& 4.9398e-12   \\
&20000 &   \bf{2.99} &  5.67&    3.15  &   0.0000e+00   & 2.0833e-13& 7.7965e-12   \\
Easy Case &30000 &   \bf{9.62} & 13.68&    9.71  &   0.0000e+00   & 4.1471e-14& 2.4518e-13   \\
&40000 &  \bf{12.88} & 17.67&   13.24  &   0.0000e+00   & 3.4451e-14& 6.7671e-13   \\
&50000 &  13.67 & 20.62&   \bf{13.65}  &   1.6836e-16   & 1.2681e-14& 1.3899e-14   \\

\hline
cond=1000&\multicolumn{4}{|c}{}\\
 \hline
 &10000 &  \bf{12.93} & 21.78&   12.98  &   0.0000e+00   & 2.5847e-11& 4.4340e-09   \\
 &20000 &  \bf{15.48} & 25.58&   17.00  &   0.0000e+00  & 1.0184e-12& 9.1688e-11   \\

 Easy case& 30000 &  \bf{21.88} & 41.75&   27.57  &   0.0000e+00   & 1.4239e-12& 5.7425e-11   \\
&40000 &  \bf{47.51} & 73.44&   54.22  &   0.0000e+00  & 4.7311e-13& 1.3561e-11   \\
&50000 &  \bf{58.08} & 98.08&   64.23  &   0.0000e+00   & 5.0026e-13& 1.1820e-11   \\

 \hline
 cond=10&\multicolumn{4}{|c}{}\\
 \hline
&10000 &   0.83 &  0.72&    \bf{0.50}  &   4.5626e-17   & 2.3592e-14& 3.3500e-14   \\
&20000 &   \bf{1.10} &  1.78&    1.18  &   2.7397e-16   & 8.3330e-15& 1.7622e-14   \\
Hard Case 1&30000 &   3.15 &  2.30&    \bf{1.57}  &   3.5323e-16   & 1.3166e-15& 4.6729e-16   \\
&40000 &   6.33 &  4.28&    \bf{3.24}  &   3.8399e-16   & 1.1407e-15& 8.2119e-16   \\
&50000 &   7.81 &  4.57&    \bf{4.61}  &   5.5852e-16   & 3.1130e-15& 5.4845e-15   \\
\hline
 cond=100&\multicolumn{4}{|c}{}\\
 \hline
 &10000 &   3.26 &  \bf{3.15}&    3.87  &   0.0000e+00   & 1.7803e-13& 4.1942e-12  \\
&20000 &   \bf{3.17} &  6.12&    3.35  &   0.0000e+00   & 1.5832e-13& 2.0746e-12   \\
Hard Case 1&30000 &   \bf{6.13} & 10.74&    6.38  &   0.0000e+00   & 2.1378e-13& 5.6080e-13   \\
&40000 &  13.63 & 14.87&   \bf{11.50}  &   0.0000e+00   & 6.9402e-14& 2.0909e-13    \\
&50000 &  18.12 & 16.61&   \bf{10.89}  &   1.7662e-17   & 3.3103e-14& 1.2959e-14     \\
\hline
cond=1000&\multicolumn{4}{|c}{}\\
 \hline
 &10000 &  13.76 & 26.63&   \bf{12.61}  &   0.0000e+00  & 1.5568e-11& 5.4200e-09   \\
&20000 &  \bf{15.46} & 27.88&   16.49  &   0.0000e+00   & 1.3622e-06& 1.3623e-06   \\
Hard Case 1& 30000 &  \bf{48.37} &102.01&   49.81  &   0.0000e+00   & 2.1438e-12& 7.6499e-11   \\
&40000 &  65.58 &105.23&   \bf{62.84}  &   0.0000e+00  & 7.4480e-13& 6.1595e-11   \\
&50000 &  80.09 &104.03&   \bf{61.66}  &   0.0000e+00  & 2.5470e-13& 5.6306e-12    \\
\hline
 cond=10&\multicolumn{4}{|c}{}\\
 \hline
&10000 &   \bf{0.77} &    2.02  &     0.98   &  0.0000e+00   & 6.3351e-14& 2.0440e-11   \\
&20000 &   \bf{1.04} &    7.51  &     1.68   &  0.0000e+00   & 2.1772e-13& 7.9143e-12   \\
Hard Case 2&30000 &   \bf{1.87} &    5.80  &     2.63   &  9.0727e-17   & 5.3708e-14& 6.4019e-14   \\
&40000 &   \bf{3.08} &   26.16  &     11.26   &  0.0000e+00  & 5.2256e-13& 3.7264e-12   \\
&50000 &   \bf{4.66} &   22.71  &     11.17   &  0.0000e+00  & 2.2534e-13& 6.8245e-13   \\
 \hline
 cond=100&\multicolumn{4}{|c}{}\\
 \hline
 &10000 &   \bf{2.50} & 16.85&    5.48  &   0.0000e+00   & 1.6785e-11& 1.2714e-08   \\
 &20000 &   \bf{3.11} &   27.54  &     6.57   &  0.0000e+00   & 1.5268e-12& 1.8101e-10   \\
Hard Case 2&30000 &   \bf{8.96}  &   75.56  &    16.39   &  0.0000e+00   & 8.8052e-12& 2.0078e-10   \\
&40000 &   \bf{9.56}&   89.83  &    20.02   &  0.0000e+00   & 1.6188e-12& 3.3824e-10   \\
&50000 &   \bf{18.37} &  173.57  &    48.50  & 0.0000e+00   & 5.3902e-12& 7.1115e-11   \\

\hline
 cond=1000&\multicolumn{4}{|c}{}\\
 \hline
& 10000 &   \bf{18.21} &   56.16  &    24.47   &  0.0000e+00  & 3.1452e-08& 9.3311e-08   \\
&20000 &   \bf{19.76} &  124.17  &    38.12   &  0.0000e+00 & 7.5228e-09& 2.1199e-08   \\
Hard Case 2&30000 &  \bf{23.91} &  222.98  &    37.13   &  0.0000e+00   & 1.4684e-09& 4.4186e-09   \\
&40000 &  \bf{46.86}  &  521.06  &    92.09   &  0.0000e+00  & 2.8725e-09& 6.5494e-09   \\
&50000 &  \bf{35.38}&  626.99  &   91.11   &  0.0000e+00   & 5.2262e-10& 4.7002e-09   \\
\hline

\end{tabular}

%% file: conjugate-08-07-2018.bbl
\begin{thebibliography}{99}
\bibitem{trs4}
S. Adachi, S. Iwata, Y. Nakatsukasa, A. Takeda,  Solving the trust region subproblem by a generalized eigenvalue problem, SIAM Journal on Optimization,  27(1), 269-291, 2017.
\bibitem{qeig}
S. Adachi, Y. Nakatsukasa, Eigenvalue-based algorithm and analysis for nonconvex
QCQP with one constraint, Mathematical Programming, DOI 10.1007/s10107-017-1206-8, 2017.
\bibitem{ben}
A. Ben-Tal and D. den Hertog, Hidden conic quadratic representation of some nonconvex
quadratic optimization problems, Mathematical Programming, 143(1-2), 1-29, 2014.
\bibitem{hat1}
C. R. Crawford, Y. S. Moon, Finding a positive definite linear combination of two Hermitian matrices, Linear Algebra and its Applications, 51, 37-48, 1983.


\bibitem{Conn} A. R. Conn, N. I. M.  Gould and  P. L. Toint,  Trust Region Methods,  SIAM, Philadelphia, PA, 2000.
 \bibitem{trsei}C. Fortin and H.  Wolkowicz, The trust region subproblem and semidefinite programming, Optimization methods and software,  19(1), 41-67, 2004.
 \bibitem{do}
 J.-M. Feng,  G.-X. Lin, R.-L. heu, Y.  Xia, Duality and solutions for quadratic programming over
single non-homogeneous quadratic constraint,  Journal of Global Optimization,  54(2), 275-293, 2012.
 \bibitem{lanczos} N. I. Gould, S. Lucidi, M.  Roma and P. L.  Toint, Solving the trust-region subproblem using the Lanczos method, SIAM Journal on Optimization, 9(2), 504-525, 1999.
      \bibitem{trs1} N. I. Gould, D. P.  Robinson and H. S.  Thorne, On solving trust-region and other regularised subproblems in optimization, Mathematical Programming Computation, 2(1), 21-57, 2010.
         \bibitem{hat2}
      C.-H.  Guo,  N.J. Higham, F. Tisseur, An improved arc algorithm for detecting definite Hermitian
pairs,  SIAM Journal on Matrix Analysis and  Applications,  31(3), 1131-1151, 2009.

   \bibitem{do1}
S. Jegelka, Private communication , 2015.
 \bibitem{novel}
          R. Jian, D. Li,  Novel reformulations and efficient algorithms for the generalized
trust region subproblem,  	arXiv:1707.08706, 2017.
\bibitem{so}
R. Jiang, D. Li and B. Wu, SOCP reformulation for the generalized trust region subproblem via a canonical form of two symmetric matrices, Mathematical Programming,  1-33, 2017.
\bibitem{co}
  P. Lancaster and L. Rodman,  Canonical forms for hermitian matrix pairs under strict equivalence and congruence, SIAM Review, 47(3), 407-443, 2005.
    \bibitem{moretrs}
J. J.  Mor\'{e} and D. C. Sorensen, Computing a trust region step, SIAM Journal on Scientific and Statistical Computing, 4, 553-572,
1983.
     \bibitem{M16} J. J.  Mor\'{e},  Generalizations of the trust region problem,  Optimization methods and Software,  2(3-4), 189-209, 1993.

         \bibitem{wol}T. K.  Pong and H. Wolkowicz, The generalized trust region subproblem,  Computational Optimization and Applications,  58(2), 273-322, 2014.
              \bibitem{trs2} F. Rendl and H. Wolkowicz,  A semidefinite framework for trust region subproblems with applications to large scale minimization,  Mathematical Programming,  77(1), 273-299, 1997.
                  \bibitem{trsei3}
M. Rojas, S. A. Santos and D. C. Sorensen,
A New Matrix-Free Algorithm for the Large-Scale Trust-Region Subproblem, SIAM Journal on Optimization, 11(3), 611-646, 2001.
\bibitem{Steihaug} T. Steihaug, The conjugate gradient method and trust regions in large scale optimization, SIAM Journal on Numerical Analysis, 20(3), 626-637, 1983.
    \bibitem{salahi}
	M. Salahi and A. Taati, An efficient algorithm for solving the generalized trust region subproblem, Computational and Applied Mathematics, 37(1), 395-413, 2018.
\bibitem{toint} P. L. Toint, Towards an efficient sparsity exploiting Newton method for minimization, in
Sparse Matrices and Their Uses, I. S. Duff, ed., London and New York, Academic Press,
57-88, 1981.
\bibitem{new}
	Y. Ye and S. Zhang, New results on quadratic minimization,  SIAM Journal on Optimization, 14(1), 245-267, 2003.
    \bibitem{8} Y. Yuan, Recent advances in trust region algorithms, Mathematical Programming, 151(1), 249-281, 2015.




\end{thebibliography}
